\documentclass[12pt]{amsart}
\usepackage{graphicx}
\usepackage{amsmath,amssymb}
\newtheorem{theorem}{Theorem}[section]

\newtheorem{lemma}[theorem]{Lemma}

\newtheorem{corollary}[theorem]{Corollary}

\theoremstyle{remark}
\newtheorem{remark}[theorem]{Remark}

\newtheorem{definition}[theorem]{Definition}

\numberwithin{equation}{section}

\begin{document}

\title[
Rectangular Seifert circles and arcs system]{
Rectangular Seifert circles and arcs system}

\author{Tatsuo Ando,
Chuichiro Hayashi
and Miwa Hayashi}

\date{\today}

\thanks{The second author is partially supported
by JSPS KAKENHI Grant Number 25400100.}

\begin{abstract}
 Rectangular diagrams of links
are link diagrams in the plane ${\mathbb R}^2$
such that
they are composed of vertical line segments and horizontal line segments
and vertical segments go over horizontal segments at all crossings.
 P. R. Cromwell and I. A. Dynnikov showed
that rectangular diagrams of links are useful 
for deciding whether a given link is split or not,
and whether a given knot is trivial or not.
 We show in this paper 
that an oriented link diagram $D$
with $c(D)$ crossings and $s(D)$ Seifert circles
can be deformed by an ambient isotopy of ${\mathbb R}^2$
into a rectangular diagram 
with at most $c(D) + 2 s(D)$ vertical segments,
and that, if $D$ is connected, at most $2c(D)+2-w(D)$ vertical segments,
where $w(D)$ is a certain non-negative integer.

 In order to obtain these results,
we show that the system of Seifert circles and arcs substituting for crossings
can be deformed by an ambient isotopy of ${\mathbb R}^2$
so that Seifert circles are 
rectangles composed 
of two vertical line segments and two horizontal line segments
and arcs are vertical line segments,
and that
we can obtain a single circle from a connected link diagram
by smoothing operations at the crossings regardless of orientation.
\end{abstract}

\keywords{
link diagram; 
rectangular diagram; 
arc presentation; 
rectangular Seifert circles and arcs system; 
number of crossings
}


\maketitle

\section{Introduction}\label{sect:introduction}

 Birman and Menasco introduced arc-presentation of links in \cite{BM},
and Cromwell formulated it in \cite{C}.
 Dynnikov pointed out in \cite{D1} and \cite{D2} 
that Cromwell's argument in \cite{C} almost shows 
that any arc-presentation of a split link
can be deformed into one which is $\lq\lq$visibly split"
by a finite sequence of elementary moves
which do not changes number of arcs of arc-presentations. 
 He also showed
that any arc-presentation of the trivial knot
can be deformed into the trivial one with only two arcs
by a finite sequence of merge elementary moves 
without increasing number of arcs.
 Since there are only finitely many arc-presentations
with a fixed number of edges,
these results give finite algorithms for the decision problems.
 As is shown in page 41 in \cite{C},
an arc-presentation is almost equivalent to a rectangular diagram.

\begin{figure}[htbp]
\begin{center}
\includegraphics[width=45mm]{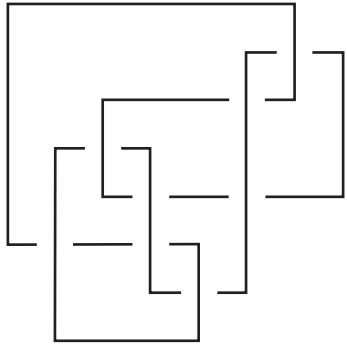}
\end{center}
\caption{A rectangular diagram of the trivial knot with $8$ vertical edges}
\label{fig:TrivialKnot8arcs}
\end{figure}

 A {\it rectangular diagram} of a link
is a link diagram in the plane ${\mathbb R}^2$
which is composed of vertical line segments and horizontal line segments 
such that no pair of vertical line segments are colinear,
no pair of horizontal line segments are colinear, 
and the vertical line segment passes over the horizontal line segment
at each crossing. 
 See Figure \ref{fig:TrivialKnot8arcs}.
 These vertical line segments and horizontal line segments 
are called {\it edges} of the rectangular diagram.
 Every rectangular diagram
has the same number of vertical edges and horizontal edges.
 It is known that every link has a rectangular diagram 
(Proposition in page 42 in \cite{C}).

 In \cite{HK}, A. Henrich and L. Kauffman announced
an upper bound of the number of Reidemeister moves
needed for unknotting (Theorem 8)
by applying Dynnikov's theorem to rectangular diagrams.  
 Let $D$ be an oriented link diagram in Morse form,
and $c(D)$, $b(D)$ the numbers of crossings and  maxima.
 Lemma 2 in \cite{HK} states 
that we can obtain a rectangular diagram 
with at most $2b(D)+c(D)$ vertical edges
from $D$ by an ambient isotopy of the plane ${\mathbb R}^2$.
 In this paper,
we consider the number of Seifert circles $s(D)$
instead of that of maxima.
 Note that $s(D)$ does not change under isotopy of ${\mathbb R}^2$.
 We obtain an estimation depending only on $c(D)$, too.

\begin{figure}[htbp]
\begin{center}
\includegraphics[width=120mm]{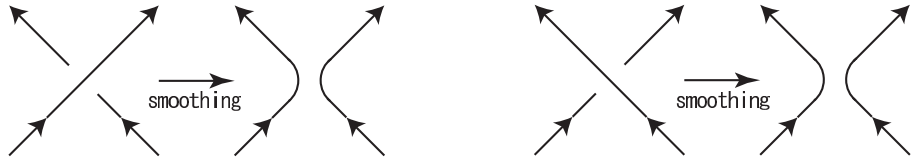}
\end{center}
\caption{smoothing operation}
\label{fig:smoothing}
\end{figure}

 Let $D$ be an oriented link diagram in the plane ${\mathbb R}^2$.
 If we perform smoothing operations at all the crossings
as shown in Figure \ref{fig:smoothing}, 
then we obtain a disjoint union of oriented circles in ${\mathbb R}^2$
as in Figure \ref{fig:Seifert},
which we call {\it Seifert circles}.
 This operation is introduced by Seifert in \cite{S}
to construct an orientable surface spanning a knot.

\begin{figure}[htbp]
\begin{center}
\includegraphics[width=80mm]{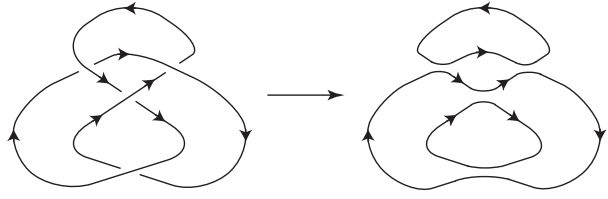}
\end{center}
\caption{Seifert circles}
\label{fig:Seifert}
\end{figure}

 A crossing $x$ of a link diagram $D$ in the plane ${\mathbb R}^2$
is called {\it nugatory}
if there is a circle $C$ in ${\mathbb R}^2$
which intersects $D$ only in a single point at $x$.
 See Figure \ref{fig:nugatory}.
 Link diagrams with a nugatory crossing are often left out from consideration
since we can get rid of a nugatory crossing
by rotating the part of the link inside $C$ through $180^{\circ}$.

\begin{figure}[htbp]
\begin{center}
\includegraphics[width=70mm]{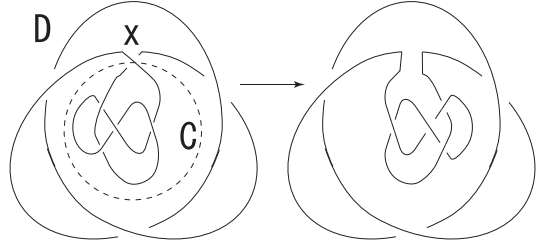}
\end{center}
\caption{a nugatory crossing}
\label{fig:nugatory}
\end{figure}

\begin{theorem}\label{theorem:c+2s}
 Let $D$ be an oriented link diagram in the plane ${\mathbb R}^2$,
and $c(D)$, $s(D)$ the numbers of crossings and Seifert circles of $D$
respectively.
 Then an adequate ambient isotopy of  ${\mathbb R}^2$
deforms $D$ into a rectangular diagram 
with at most $c(D) + 2 s(D)$ vertical edges. 
 If $c(D) \ge 1$ and $D$ has no nugatory crossings,
then $c(D) + 2 s(D) -2$ vertical edges are enough.
\end{theorem}

 A link diagram $D$ in the plane ${\mathbb R}^2$ 
is said to be {\it connected}
if it is connected 
when the underpasses are restored at all the crossings. 

 An {\it undirected smoothing operation} at a crossing 
is a smoothing operations neglecting orientation of a ink.
 It may or may not respect the orientation of the link
when we orient the link.
 By an adequate undirected smoothing operations,
as shown in Figure \ref{fig:RMCAS},
we obtain a single circle from any connected link diagram (Lemma \ref{lemma:monadic}).
 This leads to the next theorem.

\begin{theorem}\label{theorem:2c}
 Let $D$ be a connected link diagram in ${\mathbb R}^2$,
and $c(D)$ the number of crossings of $D$.
 Then $D$ can be deformed 
into a rectangular diagram with at most $2c(D) - w(D)+2$ vertical edges
by an adequate ambient isotopy of ${\mathbb R}^2$,
where $w(d)$ is a non-negative integer defined as the width of $D$ as below.
\end{theorem}

%


 In order to show Theorems \ref{theorem:c+2s} and \ref{theorem:2c}, 
we observe a system of circles obtained 
by performing smoothing operations 
at all the crossings of a link diagram
and arcs corresponding to crossings.

\begin{figure}[htbp]
\begin{center}
\includegraphics[width=110mm]{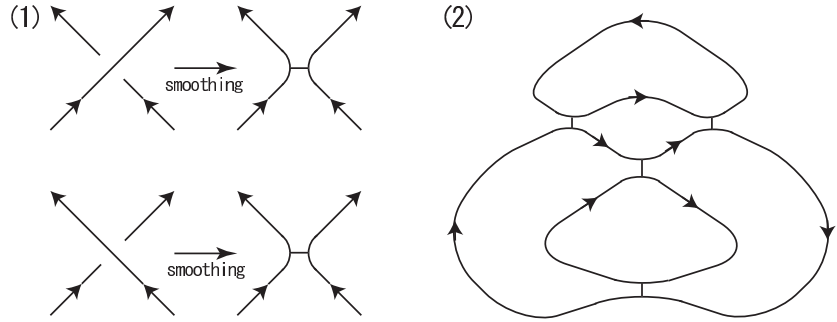}
\end{center}
\caption{Seifert circles and arcs system}
\label{fig:smoothing2}
\end{figure}

 After we perform smoothing operations at all the crossings,
we place a line segment connecting Seifert circles
as a substitute for each crossing as shown in Figure \ref{fig:smoothing2} (1).
 Then we obtain a union of circles and arcs,
which we call {\it Seifert circles and arcs system}.
 In Figure \ref{fig:smoothing2} (2),
the one which is obtained
from the knot diagram in Figure \ref{fig:Seifert} is depicted.

 Note that Seifert circles and arcs system does not have an arc
with its both endpoints in the same Seifert circle. 
(Otherwise, we would have a contradiction on orientation of the link.)
 Moreover, 
the orientations of the two circles containing the endpoints of an arc
are both clockwise or both anti-clockwise
if and only if one circle is contained in the disk bounded by the other.

%

 When we apply undirected smoothing operations,
we obtain a system of circles and arcs,
where circles do not have orientations,
and there may be an arc 
which has its both endpoints in the same circle.

 In general,
let $C$ be a disjoint union of circles in the plane ${\mathbb R}^2$,
and $A$ a disjoint union of arcs in ${\mathbb R}^2$
such that $A \cap C = \partial A$,
where $\partial A$ denotes the set of endpoints of arcs of $A$.
 Then we call the union $C \cup A$ 
{\it circles and arcs system} in ${\mathbb R}^2$.

\begin{figure}[htbp]
\begin{center}
\includegraphics[width=80mm]{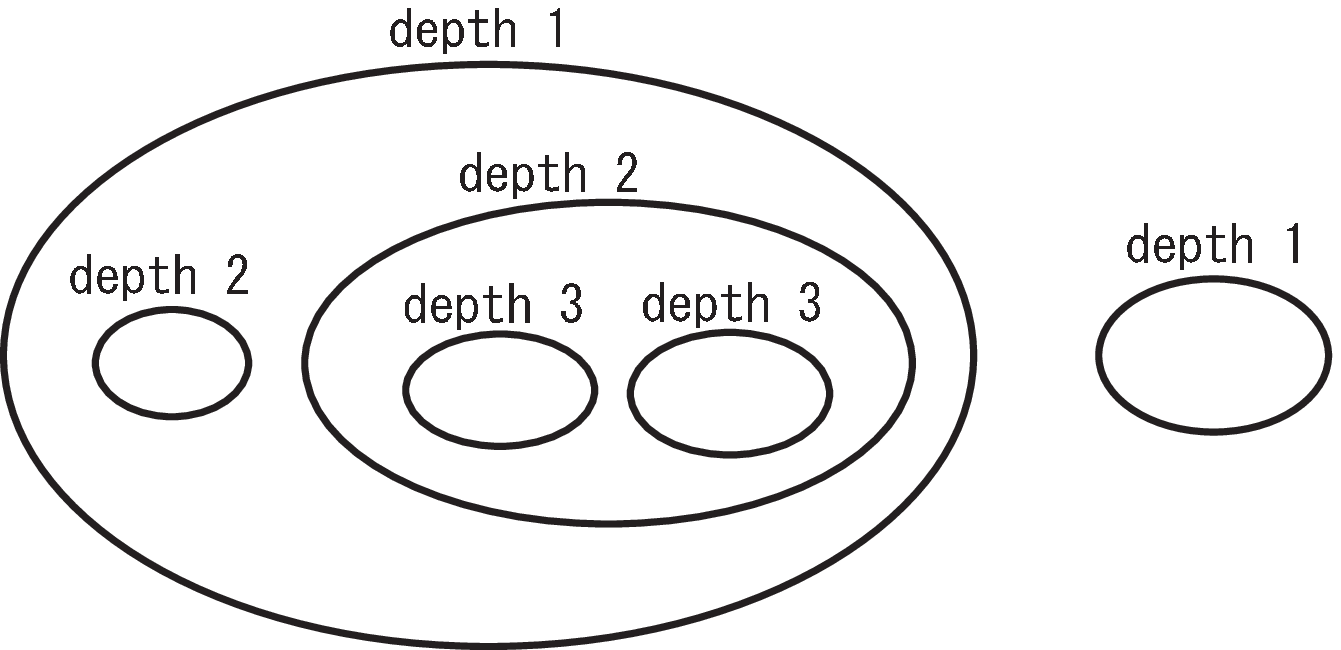}
\end{center}
\caption{depth}
\label{fig:depth}
\end{figure}

 Let $S = C \cup A$ be a circles and arcs system in ${\mathbb R}^2$.
 We divide ${\mathbb R}^2$ into regions by the circles of $C$.
 A circle $Z$ of $C$ is said to be of {\it depth} $1$
if $Z$ is contained in the boundary of the infinitely large region,
and of {\it depth} $i$
if $Z$ is contained in the boundary of a region $R$
such that its boundary $\partial R$ contains a single circle, say $Z'$, of depth $i-1$
and $R$ is inside $Z'$.
 See Figure \ref{fig:depth}.

 Every circles and arcs system can be moved to be $\lq\lq$beautiful"
as shown in the next theorem.
 Let $\pi_x: {\mathbb R}^2 \ni (x,y) \mapsto (x,0) \in {\mathbb R}^2$
and $\pi_y: {\mathbb R}^2 \ni (x,y) \mapsto (0,y) \in {\mathbb R}^2$
be projections.
 We say subsets $A$ and $B$ of ${\mathbb R}^2$
{\it overlap each other under $\pi_x$ (resp. $\pi_y$)}
if $\pi_x(A) \cap \pi_x(B) \ne \emptyset$ (resp. $\pi_y(A) \cap \pi_y(B) \ne \emptyset$).

\begin{theorem}\label{theorem:RCAS}
 Let $C \cup A$ be a circles and arcs system in ${\mathbb R}^2$,
where $C$ is the union of circles and $A$ the union of arcs.
 Then it can be deformed by an ambient isotopy of ${\mathbb R}^2$
so that 
(1) circles are rectangles composed 
of two vertical line segments and two horizontal line segments, and
(2) each arc either (a) is a vertical line segment,
or 
(b) has its both endpoints in the same circle
and is composed of three line segments:
one horizontal line segment $s$
and two vertical line segments in the same side of $s$.
 Moreover, 
the isotopy can be taken
so that no pair of rectangular circles of the same depth 
overlap each other under $\pi_y$.
\end{theorem}

 We call a circles and arcs system {\it rectangular}
if it satisfies the conditions (1) and (2) in the above theorem.
 For an example of rectangular circles and arcs system, 
see Figure \ref{fig:RCAS}.
 We say that an arc of $A$ as in (2)(a) in the theorem is of {\it type I}
and an arc as in (2)(b) is of {\it type $\sqcup$}.

\begin{figure}[htbp]
\begin{center}
\includegraphics[width=80mm]{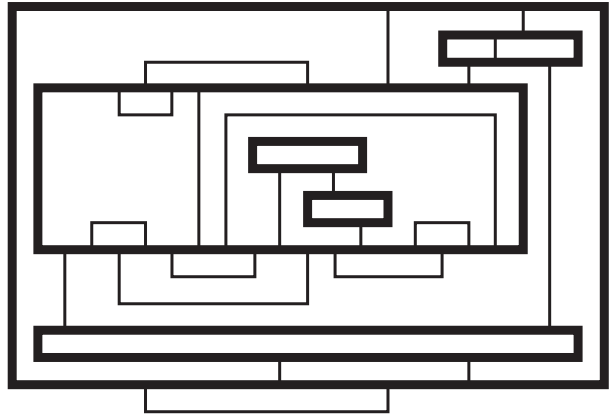}
\end{center}
\caption{rectangular circles and arcs system}
\label{fig:RCAS}
\end{figure}


\begin{corollary}\label{corollary:RSCAS}
 Let $D$ be an oriented link diagram in the plane ${\mathbb R}^2$. 
 The Seifert circles and arcs system for $D$
can be deformed by an ambient isotopy of ${\mathbb R}^2$
so that Seifert circles are 
rectangles composed 
of two vertical line segments and two horizontal line segments
and arcs are vertical line segments.
 Moreover, 
the isotopy can be taken
so that no pair of rectangular circles of the same depth 
overlap each other under $\pi_y$.
\end{corollary}

 For an example of deformation as in the above corollary, 
see Figure \ref{fig:RSCAS}.

\begin{figure}[htbp]
\begin{center}
\includegraphics[width=110mm]{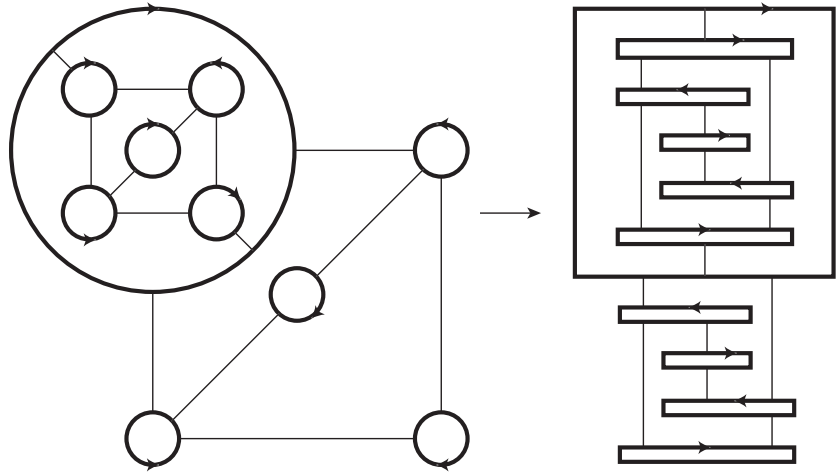}
\end{center}
\caption{deforming a Seifert circles and arcs system to be rectangular}
\label{fig:RSCAS}
\end{figure}

 A cirlces and arcs system is called {\it monadic}
if it has only one circle.
 Every connected link diagram admits a system of smoothing operations 
which yields a monadic circles and arcs system $S$.
 This is shown in Lemma \ref{lemma:monadic}.

 Let $S = C \cup A$ be a monadic circles and arcs system.
 An arc $\beta$ with $\beta \cap C = \partial \beta$ is called a {\it ruler}
if it 
is free from the endpoints $\partial A$,
is contained in the disk bounded by $C$,
and intersects every arc of $A$ transversely in at most one point.
 The number of intersection points of $\beta$ and $A$ is called the {\it length} of $\beta$.
 Then the {\it width} of $S$ is the maximal number among lengths of all rulers for $S$,
and let $w(S)$ denote it.
 The width of a connected link diagram $D$, denoted by $w(D)$, 
is the maximal width $w(S)$
over all systems of undirected smooth operations on $D$
yielding a monadic circles and arcs system $S$.

 The next is a corollary of Theorem \ref{theorem:RCAS} and Lemma \ref{lemma:monadic}.

\begin{corollary}\label{corollary:RMCAS}
 Let $D$ be a link diagram in the plane ${\mathbb R}^2$. 
 Then an adequate system of undirected smoothing operations
and an adequate ambient isotopy of ${\mathbb R}^2$
deform $D$ into a monadic circles and arcs system $C \cup A$
such that (1) the circle $C$ is a rectangle composed 
of two vertical line segments and two horizontal line segments,
and (2) each arc of $A$ is of type I or $\sqcup$.
\end{corollary}

 For an example of deformation as in the above corollary, 
see Figure \ref{fig:RMCAS}.

\begin{figure}[htbp]
\begin{center}
\includegraphics[width=120mm]{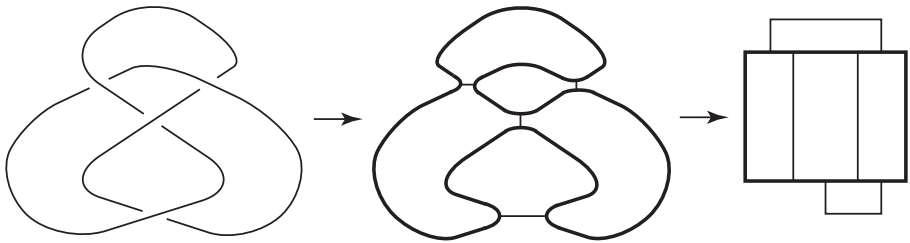}
\end{center}
\caption{deforming a link diagram to a rectangular monadic circles and arcs system}
\label{fig:RMCAS}
\end{figure}

\section{Proof of Theorem \ref{theorem:2c}}\label{section:2c}

\begin{lemma}\label{lemma:monadic}
 From any connected link diagram $D$,
we can obtain a monadic circles and arcs system
by adequately applying undirected smoothing operations
at all the crossings of $D$.
\end{lemma}

\begin{figure}[htbp]
\begin{center}
\includegraphics[width=70mm]{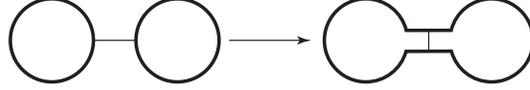}
\end{center}
\caption{changing the way of undirected smoothing at a crossing}
\label{fig:MakeMonadic}
\end{figure}

\begin{proof}
 We apply arbitrary undirected smoothing operations at all the crossings of $D$,
to obtain a circles and arcs system $S = C \cup A$.
 Note that $S = C \cup A$ is connected because $D$ is connected.
 If the set of circles $C$ consists of a single circle, we are done.
 If it contains plural circles,
then there is an arc of $A$ which connects two distinct circles of $C$.
 We change the way of undirected smoothing at the crossing corresponding the arc,
and obtain a circles and arcs system with one less circles.
 See Figure \ref{fig:MakeMonadic}.
 Repeating this, we obtain a monadic circles and arcs system.
\end{proof}

\begin{figure}[htbp]
\begin{center}
\includegraphics[width=125mm]{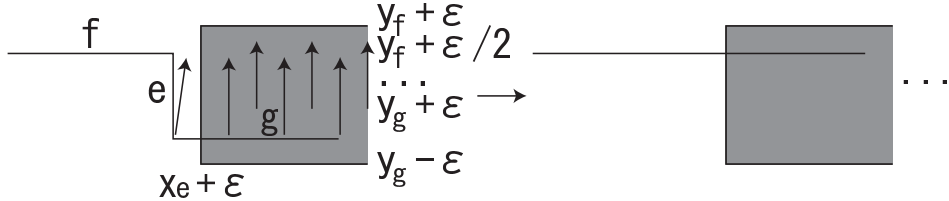}
\end{center}
\caption{straight merge}
\label{fig:s-merge}
\end{figure}

\begin{lemma}\label{lemma:s-merge}
 Let $R$ be a rectangular diagram with $n$ vertical edges.
 Suppose that $R$ has a vertical edge $e$
such that $e$ is free from the crossings of $R$,
and that the horizontal edges $f, g$ sharing an endpoint with $e$
are in the opposite sides of $e$ to each other.
 Then an adequate ambient isotopy of ${\mathbb R}^2$
deforms $f \cup e \cup g$ into a single horizontal edge,
and $R$ into a rectangular diagram with $n-1$ vertical edges.
\end{lemma}

 We call the ambient isotopy as in the proof below a {\it straight merge} operation at $e$.
 A similar thing holds also when $e$ is horizontal.

\begin{proof}
 Let $y_f$, $y_g$ be the ordinates of $f$, $g$ respectively,
and $x_e$ the abscissa of $e$.
 We assume, without loss of generality, that $y_f > y_g$, 
and that $f$ is in the left side of $e$ and $g$ is in the right side of $e$.
 We take a small positive real number $\epsilon$
so that there is no horizontal edge of $R$ 
at any ordinate in $(y_g -\epsilon, y_g) \cup (y_g, y_g+\epsilon) \cup (y_f, y_f+\epsilon)$
and that there is no vertical edge of $R$
at any abscissa in $(x_e, x_e+\epsilon)$.
 Then there is an ambient isotopy of ${\mathbb R}^2$ as in the conclusion of this lemma
which shrinks
the area $[y_g+\epsilon, y_f+\epsilon] \times [x_e+\epsilon, \infty)$
in the vertical direction
so that it is deformed
into the area $[y_f+(\epsilon/2), y_f+\epsilon] \times [x_e+\epsilon, \infty)$
and expands
the area $[y_g-\epsilon, y_g+\epsilon] \times [x_e+\epsilon, \infty)$
in the vertical direction
so that it is deformed
into the area $[y_g-\epsilon, y_f+(\epsilon/2)] \times [x_e+\epsilon, \infty)$.
 See Figure \ref{fig:s-merge}.
\end{proof}

\begin{figure}[htbp]
\begin{center}
\includegraphics[width=70mm]{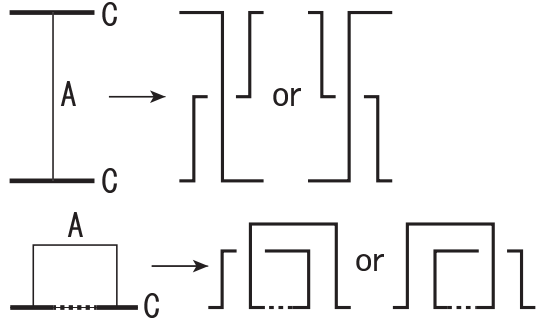}
\end{center}
\caption{restoring the link diagram}
\label{fig:restore}
\end{figure}

\begin{proof}
 We prove Theorem \ref{theorem:2c}.
 Let $D$ be a connected link diagram,
and $c(D)$ the number of crossings of $D$.
 We perform undirected smoothing operations at all the crossings
so that we obtain a monadic circles and arcs system $S=C\cup A$
with $w(D) = w(S)$.

 Let $\beta$ be a ruler for $S$
such that its length gives the width $w(D)$.
 We can isotope $S \cup \beta$ 
so that (1) $\beta$ is a horizontal line segment,
that (2) the circle $C$ forms the boundary circle of a tubular neighbourhood of $\beta$
and consists of two horizontal line segments and two vertical line segments,
and that (3) arcs intersecting $\beta$ in a single point are vertical line segments
and the other arcs of $A$ 
consist of two vertical line segments and a single horizontal line segment.

 Then we restore the diagram $D$ from $S$
by replacing arcs of $A$ with crossings as shown in Figure \ref{fig:restore}.
 For each of $w(D)$ vertical arcs of $A$,
we use a vertical line segment as an overpass, 
and a union of two vertical line segments and a single horizontal line segment as an underpass
to form a crossing.
 Note that the two vertical lines are free from the crossing,
and horizontal lines with which they share endpoints are in the opposite sides of them.
 For each of the other arcs of $A$, 
we use a union of two vertical line segments and a single horizontal line segment
as both of an overpass and an underpass.
 Note that the three vertical lines are free from the crossing.
 For two of them,
horizontal lines with which they share endpoints are in the opposite sides of them.
 Thus we obtain a link diagram
composed of 
$2+3w(D)+4(c(D)-w(D))=2+4c(D)-w(D)$ vertical line segments.

\begin{figure}[htbp]
\begin{center}
\includegraphics[width=70mm]{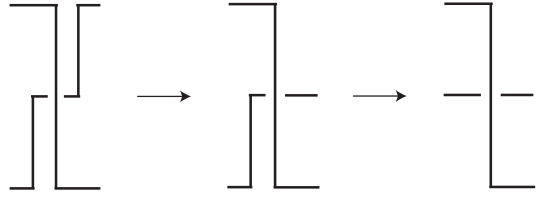}
\end{center}
\caption{straight merge operations at vertical line segments
in an underpass}
\label{fig:SmergeAtVerticalLine}
\end{figure}

 By slightly perturbing ordinates of horizontal line segments,
and abscissa of vertical line segments,
we obtain a rectangular diagram with 
$2+4c(D)-w(D)$ vertical edges.
 Then we perform straight merge operations at vertical edges as above,
twice per each crossing
as shown in Figure \ref{fig:SmergeAtVerticalLine},
to obtain a rectangular diagram 
with $(2+4c(D)-w(D))-2c(D)=2+2c(D)-w(D)$ vertical edges.
\end{proof}

\begin{figure}[htbp]
\begin{center}
\includegraphics[width=80mm]{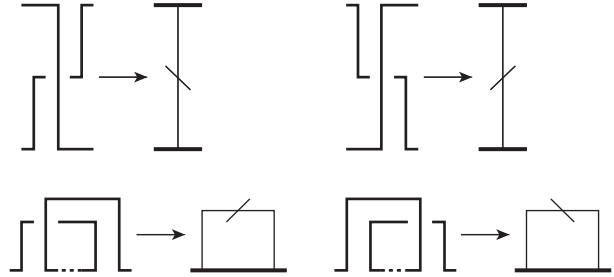}
\end{center}
\caption{Crossings are denoted by arcs with a slash.}
\label{fig:notation}
\end{figure}

 In the rest of this section,
crossings are denoted by arcs with a slash
as shown in Figure \ref{fig:notation}
when we need to specify 
over or under information of them.

\begin{figure}[htbp]
\begin{center}
\includegraphics[width=100mm]{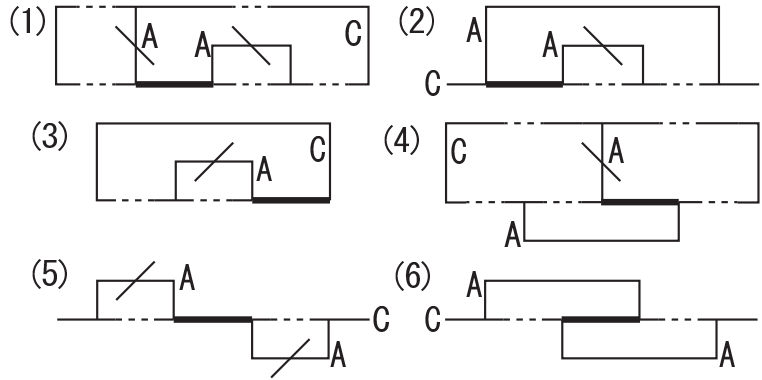}
\end{center}
\caption{horizontal line segments at which we perform straight merge operations}
\label{fig:MergeAtHorizontalEdge}
\end{figure}

 After applying straight merge operations at vertical line segments
twice per each crossing as in the above proof,
we can perform straight merge operations
at certain kinds of horizontal line segments
if there are.
 Let $C_h$ be the union of the top and bottom horizontal line segments of $C$.
 The endpoints of $A$ divide $C_h$ into shorter horizontal line segments.
 Among them,
those which are depicted in bold lines in Figure \ref{fig:MergeAtHorizontalEdge}
are available for straight merge operations.
 For type (5), for example, 
the operation is performed as shown in Figure \ref{fig:MergeAtHorizontalEdge2}.

\begin{figure}[htbp]
\begin{center}
\includegraphics[width=120mm]{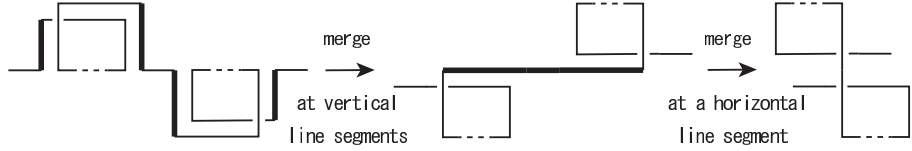}
\end{center}
\caption{a straight merge at a horizontal line segment}
\label{fig:MergeAtHorizontalEdge2}
\end{figure}

 In case of the circles and arcs system in Figure \ref{fig:cannot},
we cannot apply straight merge operation
at any horizontal line segment
after the straight merge operations at vertical line segments.
 (However, the resulting rectangular diagram is not minimal
with respect to the number of vertical edges.)

\begin{figure}[htbp]
\begin{center}
\includegraphics[width=120mm]{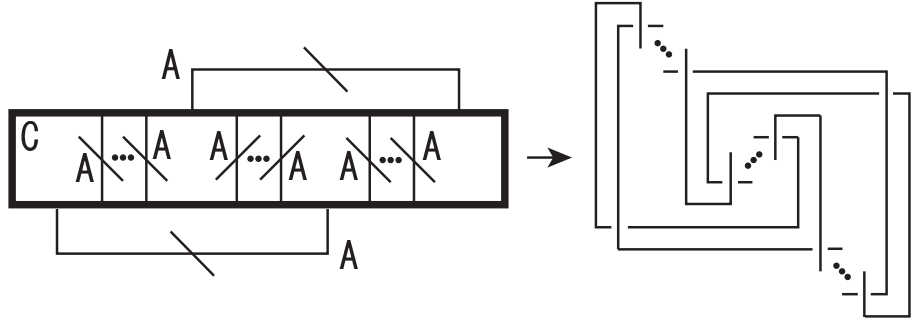}
\end{center}
\caption{We cannot apply a straight merge at any horizontal line segment to this system.}
\label{fig:cannot}
\end{figure}

 If there is an arc of type $\sqcup$ inside $C$,
and the link diagram does not have a monogon face,
then we can apply a straight merge at a horizontal line segment
as in the lemma below.
 Let $\alpha$ be an arc of type $\sqcup$ inside $C$ (resp. outside $C$),
and $Q$ the disk bounded by the circle
formed by $\alpha$ and a subarc of $C_h$.
 We call $\alpha$ is {\it innermost}
among the arcs of type $\sqcup$ inside $C$ (resp. outside $C$)
if $Q$ does not contain such an arc other than $\alpha$.

\begin{lemma}\label{lemma:InnermostSqcup}
{\rm 
 Let $D, S, C, A$ as in the proof of Theorem \ref{theorem:2c},
$C_h$ as above,
and $m$ the number of innermost arcs of type $\sqcup$ inside $C$.
 Suppose that $D$ has no monogon region.
 After performing straight merge operations
at vertical line segments twice per each arc of $A$,
we can apply at least $m$ straight merge operations
at hirozontal line segments as in Figure \ref{fig:MergeAtHorizontalEdge} (6).
}
\end{lemma}

\begin{proof}
 Let $\alpha$ be an arc of type $\sqcup$ which is innermost inside $C$,
and $\beta$ the subarc of $C_h$ between the endpoints $\partial \alpha$.
 If int\,$\beta$ is disjoint from endpoints of $A$, 
then $D$ has a monogon region, a contradiction.
 Hence int\,$\beta$ contains an endpoint of $A$.
 Let $\gamma$ be an innermost arc of $A$ outside $C$
with at least one of its endpoints in $\beta$.
 If $\gamma$ has both its endpoints in $\beta$,
then $\gamma$ gives a monogon region of $D$,
which is a contradiction.
 Hence $\gamma$ has precisely one of its endpoints in $\beta$.
 Then $\alpha$ and $\gamma$ together form the pattern (6) in Figure \ref{fig:MergeAtHorizontalEdge},
and we can apply a straight merge operation there.
\end{proof}

\section{Proof of Theorem \ref{theorem:c+2s}}\label{section:c+2s}

%
%

In this section, 
we prove Theorem \ref{theorem:c+2s}
using Corollary \ref{corollary:RSCAS}.
(Corollary \ref{corollary:RSCAS} is a corollary of Theorem \ref{theorem:RCAS}
which is proven in the next section.)

\begin{proof}
 We prove Theorem \ref{theorem:c+2s}.
 Let $D$ be an oriented link diagram, 
and $c(D)$ and $s(D)$ the numbers of crossings and Seifert circles of $D$
respectively.
 We perform smoothing operations at all the crossings, 
so that we obtain a Seifert circles and arcs system $S=C \cup A$.
 Using Corollary \ref{corollary:RSCAS},
we can deform $S$
by an ambient isotopy of ${\mathbb R}^2$
into a rectangular Seifert circles and arcs system $R$.
 We restore the diagram $D$ from $R$
by replacing arcs of $A$ with crossings
as shown in the upper half of Figure \ref{fig:restore}.
 The resulting link diagram has $2s(D)+3c(D)$ vertical line segments.
 Then we perform straight merge operations
as described in the proof of Lemma \ref{lemma:s-merge}
twice per each crossing
as shown in Figure \ref{fig:SmergeAtVerticalLine}.
 Note that each straight merge operation
decreases the number of vertical line segments by one.
 After an adequate small ambient isotopy of ${\mathbb R}^2$,
no pair of line segments are colinear.
 Thus we obtain a rectangular diagram with $2s(D)+3c(D)-2c(D)=2s(D)+c(D)$ vertical edges.



\begin{figure}[htbp]
\begin{center}
\includegraphics[width=120mm]{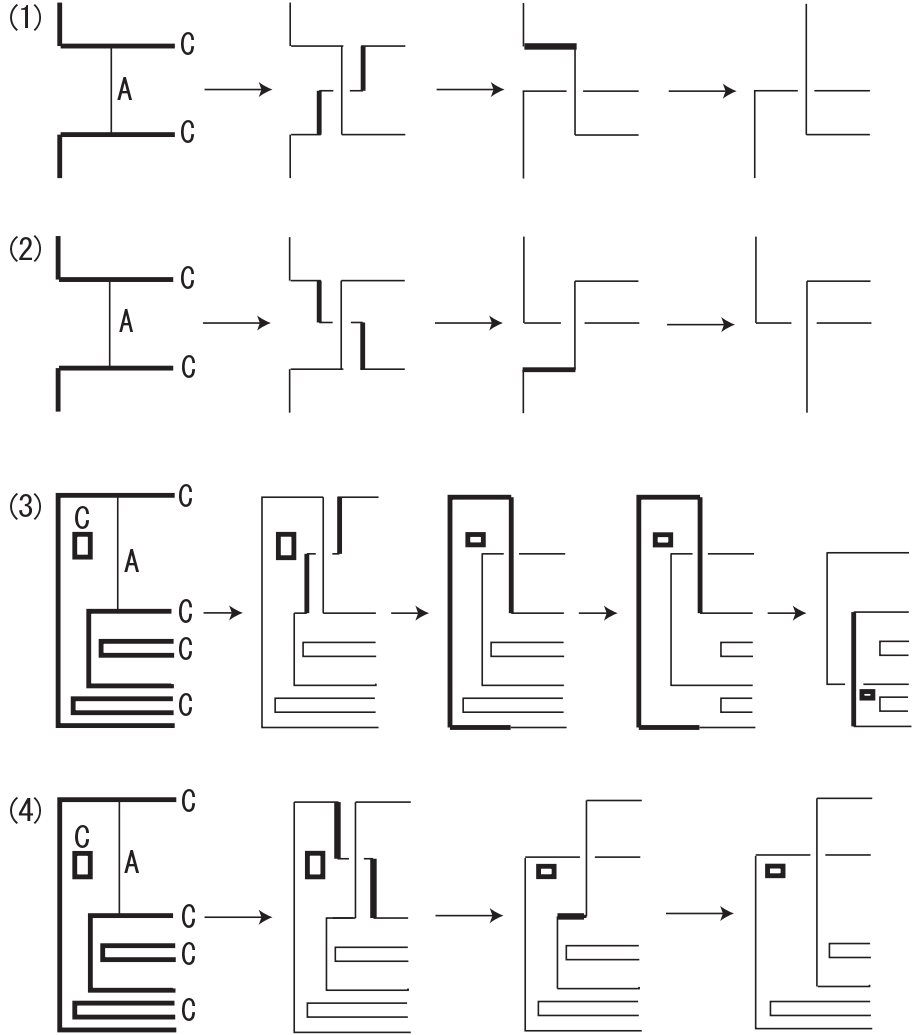}
\end{center}
\caption{straight merge or generalized merge}
\label{fig:MergesAtLeftmostArc}
\end{figure}
 
 We consider the case
where $c(D) \ge 1$ and $D$ has no nugatory crossings.
 If $c(D)=1$, then $D$ would have a nugatory crossing.
 Hence $c(D) \ge 2$.
 We observe the leftmost arc $\alpha_l$ of $A$.
 In the argument in the previous paragraph,
we have performed straight merge operations twice
at two vertical line segments in underpass of the crossing
corresponding to $\alpha_l$.
 The second arrows in Figure \ref{fig:MergesAtLeftmostArc}
show the deformation as above.
 Now, we perform deformations
described by the third and the forth arrows in Figure \ref{fig:MergesAtLeftmostArc}.
 Note that $\alpha_l$ has its endpoints in distinct circles of $C$
since $R$ is a Seifert circles and arcs system.
 In Cases (1) and (2) in Figure \ref{fig:MergesAtLeftmostArc},
$\alpha_l$ has its endpoints in circles of the same depth,
while it does not in Cases (3) and (4).
 In Cases (1), (2) and (4) in Figure \ref{fig:MergesAtLeftmostArc},
we apply a straight merge operation at a horizontal line segment drawn in a bold line.
 In Case (3) in Figure \ref{fig:MergesAtLeftmostArc},
we shrink circles of $C$
which overlap $\alpha_l$ under $\pi_x$ and do not contain an endpoint of $\alpha_l$
to the right direction
so that they are in the right side of the vertical line segment
forming the overpass of the crossing corresponding $\alpha_l$,
and then,
we perform a $\lq\lq$generalized merge operation"
as shown by the forth arrow.
 We do a similar operation for the rightmost arc of $A$,
to obtain a rectangular diagram with $2s(L)+c(L)-2$ vertical edges. 
\end{proof}

\section{Rectangular circles and arcs system}\label{section:RCAS}

 We prove Theorem \ref{theorem:RCAS} in this section.
%

\begin{definition}\label{definition:BlackBox}
 Let $S= C \cup A$ be a rectangular circles and arcs system
with $C$ being the union of circles and $A$ the union of arcs.
 Let $R$ be a rectangular disk
whose boundary circle $\partial R$ is composed of two vertical line segments
and two horizontal line segments.
 We call $S \cap R$ a {\it black box} in $S$
if $\partial R \cap C = \emptyset$,
the vertical lines of $\partial R$ are disjoint from $S$,
the horizontal lines of $\partial R$ intersect $A$ transversely
in one or more points.
\end{definition}

\begin{figure}[htbp]
\begin{center}
\includegraphics[width=60mm]{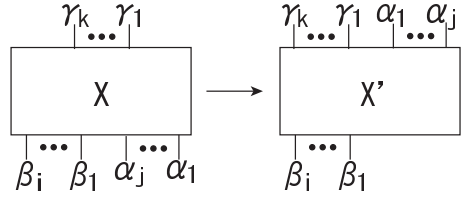}
\end{center}
\caption{Deformation in Lemma \ref{lemma:BlackBox}}
\label{fig:BlackBoxLemma}
\end{figure}

\begin{lemma}\label{lemma:BlackBox}
 Let $S, C, A, R$ be as in Definition \ref{definition:BlackBox} above. 
 Let $\tilde{R}$ be a rectangular disk
such that it forms a regular neighbourhood of $R$,
that its boundary circle $\partial{\tilde{R}}$
is composed of two vertical line segments and two horizontal line segments,
and that $S \cap (\tilde{R}-{\rm int}\,R)$ consists of vertical line segments,
say, $\gamma_1, \gamma_2, \cdots, \gamma_k$ from the right
above the top horizontal line segment of $\partial R$,
and $\alpha_1, \alpha_2, \cdots, \alpha_j, \beta_1, \beta_2, \cdots, \beta_i$
from the right
below the bottom horizontal line segment of $\partial R$
for some non-negative integers $i,k$ and a positive integer $j$.
 Suppose that there is no vertical line segment in $A$ 
connecting
one of $\gamma_1, \cdots, \gamma_k$
and one of $\alpha_1, \cdots, \alpha_j, \beta_1, \cdots, \beta_i$,
and that no pair among $\alpha_1, \cdots, \alpha_j, \beta_1, \cdots, \beta_i$ is contained 
in the same arc of type $\sqcup$.
 See Figure \ref{fig:BlackBoxLemma}.
 Then there is an ambient isotopy $H: {\mathbb R}^2 \times [0,1] \rightarrow {\mathbb R}^2$
with a homeomorphism $H_t : {\mathbb R}^2 \ni x \mapsto H_t(x)=H((x,t)) \in {\mathbb R}^2$
fot all $t \in [0,1]$
such that $H_t (R) = R$ and $H_t(\tilde{R})=\tilde{R}$ for all $t \in [0,1]$,
that $H_t (p) = p$ 
for all points $p$ in the left vertical line segments of $\partial R \cup \partial \tilde{R}$
and for all $t \in [0,1]$,
that $H_1(\alpha_m)$ is a vertical line segment
above the top horizontal line segment of $\partial R$
for all $m \in \{ 1, 2, \cdots, j \}$,
and that $H_1(S \cap R)$ is a black box
in some rectangular circles and arcs system.
\end{lemma}

\begin{remark}
 This lemma does not mention deformation of $S$ outside $\tilde{R}$
which should occur accompanied by the deformation within $\tilde{R}$.
 The isotopy in this lemma may not keep $S$ rectangular.
\end{remark}

\begin{proof}
 We say that
an ambient isotopy $H: {\mathbb R}^2 \times [0,1] \rightarrow {\mathbb R}^2$ is {\it good}
if $H_t (R) = R$ and $H_t(\tilde{R})=\tilde{R}$ for all $t \in [0,1]$,
and $H_t (p) = p$
for all points $p$ in the vertical line segments of $\partial R \cup \partial \tilde{R}$
and for all $t \in [0,1]$.

 We can move the circles and arcs system $S$
by a good ambient isotopy of ${\mathbb R}^2$
so that the rectangular circles of $C$ are thinned in the vertical direction,
and no pair of rectangular circles of the same depth overlap under $\pi_y$,
since the arcs of $A$ intersects $C$
in endpoints of vertical line segments in $A$.

 It is enough to show this lemma in the case $j=1$.
 Applying the result for the case $j=1$ repeatedly, 
we obtain the desired conclusion also for the case $j>1$.

 Set $X = S \cap R$, the black box.
 The vertical line segment of $A$ containing $\alpha_1$ has an endpoint $p$
in the bottom horizontal line segment $b$ 
of some circle $Z$ of $C$ in $X$.
 The point $p$ divides $b$ into two segments,
one of which, say $b_r$, lies in the right side of $p$.
 Let $t$ be the top horizontal line segment of $Z$.

\begin{figure}[htbp]
\begin{center}
\includegraphics[width=120mm]{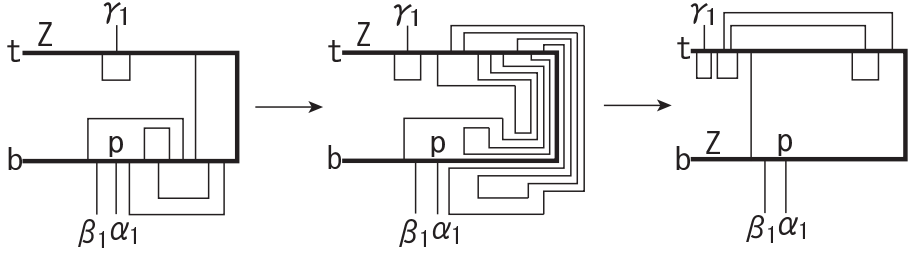}
\end{center}
\caption{The case $c(X)=1$}
\label{fig:nCX1}
\end{figure}

 Let $n_C(X)$ denote the number of circles of $C \cap X$.
 The proof proceeds by induction on $n_C(X)$.
 We first consider the case $n_C(X)=1$. 
 Then $Z$ is the only circle of $C$ contained in $X$.
 We move arcs of $A$ by a good ambient isotopy of ${\mathbb R}^2$.
 If the interior of $b_r$ contains endpoints of arcs of $A$,
we move the arcs near the endpoints
along a subarc of $b_r$ and the right vertical line segment of $Z$
so that the endpoints are contained in the top horizontal line segment of $Z$.
 See Figure \ref{fig:nCX1}.
 Because $j=1>0$, 
no arc of $A$ outside $Z$ has 
an endpoint in $t$
and the other in $b-b_r$ or below the bottom horizontal line segment of $\partial R$.
 Any arc with both endpoints in $t$
cobounds a disk with a subarc of $t$ inside or outside $Z$.
 Then we move arcs of $A$
so that $A \cap R$ is a union of arcs of types I and $\sqcup$.
 There is no obstruction because $n_C(X)=1$.
 Thus we can assume
that int\,$b_r$ does not contain such an endpoint of $A$.
 Then we can move $\alpha_1$ so that it is above the top horizontal line segment of $Z$
as shown in Figure \ref{fig:nCX1a1},
which shows the lemma in the case of $n_C(X)=1$.

\begin{figure}[htbp]
\begin{center}
\includegraphics[width=120mm]{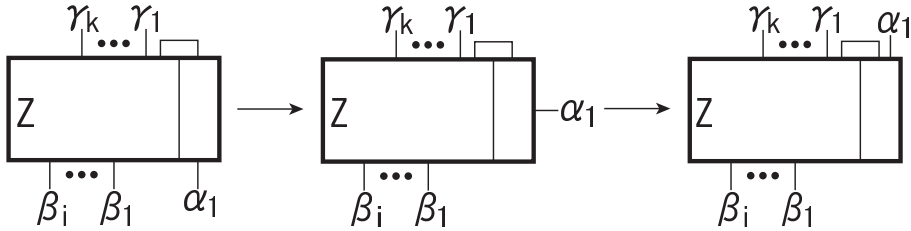}
\end{center}
\caption{Bring $\alpha_1$ above $t$}
\label{fig:nCX1a1}
\end{figure}

 We consider the case $n_C(X) > 1$.
 Let $y_t, y_b$ (resp. $\eta_t, \eta_b$) be 
the ordinates of the top and the bottom horizontal line segments of $Z$ (resp. $\partial R$),
and $x_l, x_r$ (resp. $\xi_l, \xi_r$)
the abscissae of the left and the right vertical line segments of $Z$ (resp. $\partial R$).
 Let $x_{\alpha}$ be the abscissa of $\alpha_1$.
 We consider rectangles
\newline 
$R_a=[\xi_l +\epsilon, \xi_r -\epsilon] \times [y_t +\epsilon, \eta_t -\epsilon]$
and
$R_b=[x_{\alpha} +\epsilon, \xi_r -\epsilon] \times [\eta_b +\epsilon, y_b -\epsilon]$,
and black boxes $X_a = S \cap R_a$ and $X_b = S \cap R_b$,
where $\epsilon$ is a small positive real number
such that there is no horizontal line segment in $S$ with ordinate
in the union of open intervals
$(\eta_b, \eta_b +\epsilon) \cup (y_b -\epsilon,y_b)
\cup (y_t, y_t +\epsilon) \cup (\eta_t -\epsilon, \eta_t)$
and there is no vertical line segment in $S$ with abscissa in
$(\xi_l, \xi_l +\epsilon) \cup (x_{\alpha}, x_{\alpha}+\epsilon)
\cup (\xi_r -\epsilon, \xi_r)$.

\begin{figure}[htbp]
\begin{center}
\includegraphics[width=130mm]{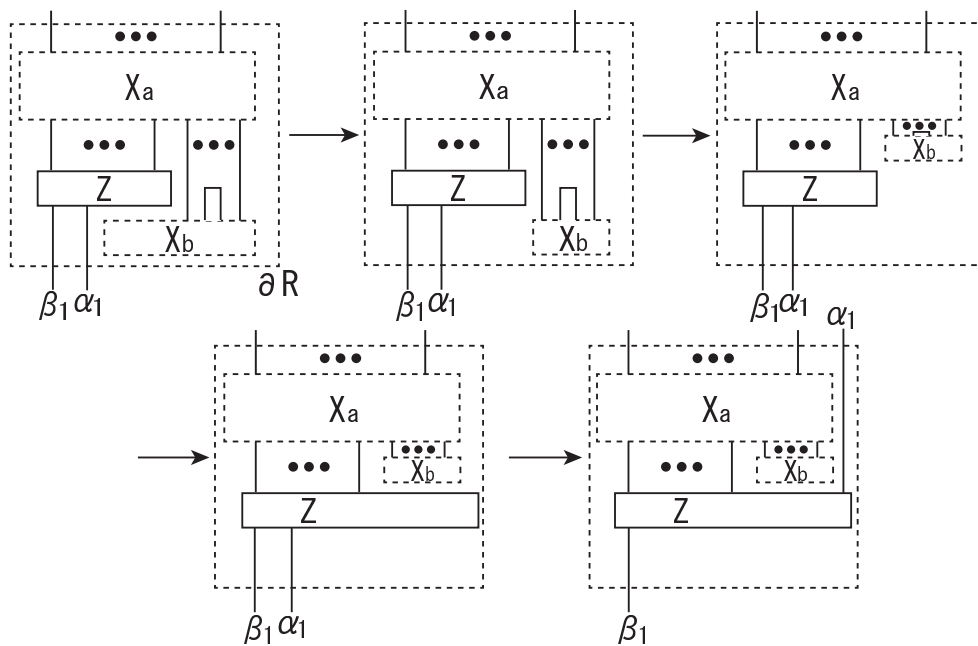}
\end{center}
\caption{The case where int\,$b_r$ is free from endpoints}
\label{fig:ne0}
\end{figure}

 We consider first the case
where int\,$b_r$ does not contain an endpoint of an arc of $A$ outside $Z$.
 In this case, 
the proof proceeds
by induction on the number, say $n_e$, of endpoints of arcs of $A$ contained in int\,$b_r$.
 Such arcs are inside $Z$.
 When $n_e=0$,
we deform $S$ as in Figure \ref{fig:ne0}.
 We shrink $X_b$ in the horizontal direction 
by a good ambient isotopy of ${\mathbb R}^2$
so that it is contained in $(x_r, \xi_r -\epsilon] \times [\eta_b + \epsilon, y_b - \epsilon]$.
 Then, shrinking the arcs connecting $X_b$ and $X_a$ and thinning $X_b$ in the vertical direction,
we move $X_b$ upward so that the ordinates of points in $X_b$ are within the interval $(y_t, y_t +\epsilon)$.
 Now, there is nothing in the right side of $Z$ within $R$.
 We lengthen $Z$ to the right direction
so that the abscissa of the right vertical line segment of $Z$
is a little larger than $\xi_r - \epsilon$.
 Then we can move the arc $\alpha_1$
so that it forms a vertical line segment above the top horizontal line segment of $Z$.

\begin{figure}[htbp]
\begin{center}
\includegraphics[width=80mm]{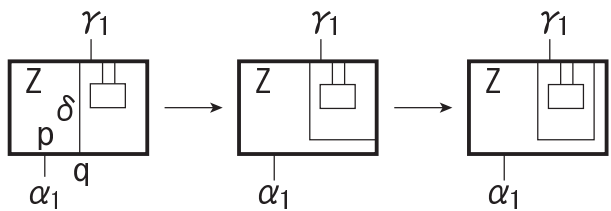}
\end{center}
\caption{The case where $\delta$ is of type I and with $\partial \delta$ in $Z$}
\label{fig:DeformTypeI}
\end{figure}

 We consider the case $n_e > 0$.
 Let $q$ be the rightmost endpoint among those of arcs of $A$ in int\,$b_r$,
and $\delta$ the arc with $q \in \partial \delta$.  
 Note that $\delta$ is inside $Z$.
 If $\delta$ is an arc of type I with its both endpoints in $Z$,
then we move $\delta$ near the point $q$
along a subarc of $b_r$ and the right vertical line segment of $Z$,
so that $\delta$ is deformed into an arc of type $\sqcup$.
 See Figure \ref{fig:DeformTypeI}.

\begin{figure}[htbp]
\begin{center}
\includegraphics[width=120mm]{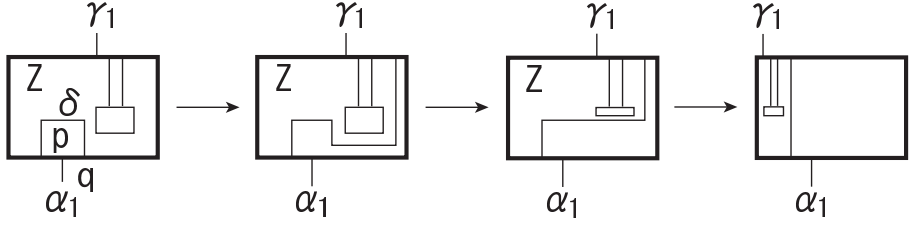}
\end{center}
\caption{The case where $\delta$ is of type $\sqcup$}
\label{fig:DeformTypeSqcup}
\end{figure}

 If $\delta$ is an arc of type $\sqcup$,
then 
we can move $\delta$ as shown in Figure \ref{fig:DeformTypeSqcup}
so that it forms an arc of type I.
 We deform $\delta$ near the point $q$
along a subarc of $b_r$ and the right vertical line segment of $Z$,
and then perform deformations similar to straight merges 
as described in the proof of Lemma \ref{lemma:s-merge}. 

\begin{figure}[htbp]
\begin{center}
\includegraphics[width=100mm]{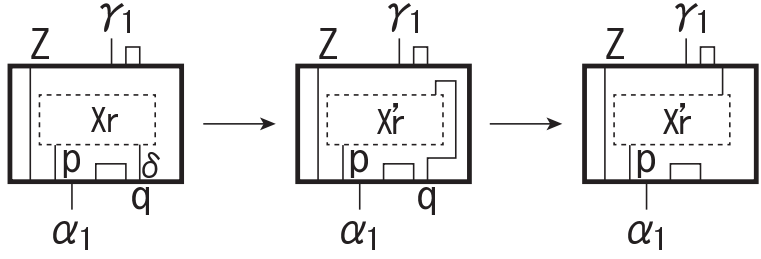}
\end{center}
\caption{The case where $\delta$ connects $Z$ and other circle of $C$}
\label{fig:DeformTypeI2}
\end{figure}

 In the remaining case,
$\delta$ has the other endpoint in a circle of $C$ other than $Z$.
 Let $R_Z$ be the rectangular disk bounded by $Z$ in ${\mathbb R}^2$.
 The arcs of type I with their both endpoints in $Z$ divides $R_Z$ into subdisks.
 Let $R'_r$ be the rightmost one
with $R'_r = [x_I, x_r] \times [y_b, y_t]$ for some real number $x_I$.
 If there are no such arc of type I, 
then we set $x_I = x_l$ and $R'_r = R_Z$.
 Then, set 
$R_r = [x_I +\epsilon', x_r -\epsilon'] \times [y_b +\epsilon', y_t -\epsilon']$,
where the positive real number $\epsilon'$ is taken to be small
so that there is no horizontal line segment in $S$ with ordinate in
$(y_b, y_b +\epsilon') \cup (y_t -\epsilon', y_t)$
and there is no vertical line segment in $S$ with abscissa in
$(x_I, x_I +\epsilon') \cup (x_r -\epsilon', x_r)$.
 Let $X_r = S \cap R_r$, the black box.
 If $X_r$ contains an arc $\lambda$ of type $\sqcup$ with its both endpoints in $b$,
then we deform the rectangle cobounded by $\lambda$ and a subarc of $b$
to be very thin in the vertical direction
so that $\lambda$ gets out from $X_r$. 
 See Figure \ref{fig:DeformTypeI2}.
 Note that $n_C(X_r) < n_C(X)$. 
 By the hypothesis of induction,
we can apply Lemma \ref{lemma:BlackBox} to $X_r$. 
 Then the black box $X_r$ is deformed to some black box, say $X'_r$,
and the arc $\delta$ is deformed to an arc
which connects $b_r$ and the top horizontal line segment of $\partial R_r$.
 Then $\delta$ can be deformed to a vertical line segment
connecting the top horizontal line segment of $Z$
and that of $\partial R_r$.

 In each of the three cases above we can decrease $n_e$,
and the lemma follows by induction.

\begin{figure}[htbp]
\begin{center}
\includegraphics[width=100mm]{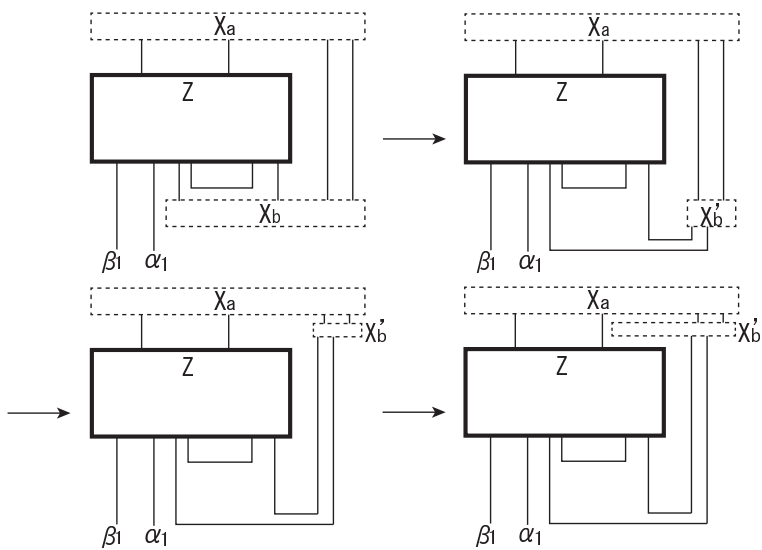}
\end{center}
\caption{The case where $\delta$ is outside $Z$}
\label{fig:outsideZ1}
\end{figure}

 We consider the case 
where the interior of $b_r$ contains an endpoint of an arc of $A$ outside $Z$.
 If $X_b$ contains an arc $\mu$ of type $\sqcup$ with its both endpoints in int\,$b_r$,
then we deform the rectangle cobounded by $\mu$ and a subarc of $b_r$
to be very thin in the vertical direction
so that $\mu$ gets out from $X_b$. 
 Note that $n_C(X_b) < n_C(X)$. 
 The hypothesis of induction allows us 
to apply Lemma \ref{lemma:BlackBox} to $X_b$
so that the vertical line segments connecting $b$ and the top line segment of $\partial R_b$
are deformed to be arcs connecting $b$ and the bottom line segment of $\partial R_b$. 
 Let $X'_b$ be the black box obtained from $X_b$ by this deformation.
 See Figure \ref{fig:outsideZ1}.
 We move $X'_b$ as in this figure.
 We shrink $X'_b$
so that the abscissa of one of the leftmost points of $X'_b$ is 
larger than that of the right vertical line segment of $Z$,
and lift $X'_b$ up
so that the ordinate of one of the bottom points of $X'_b$ is
larger than that of the top horizontal line segment of $Z$,
and then lengthen it
so that the abscissa of one of the leftmost points of $X'_b$ is
smaller than that of the right vertical line segment of $Z$.
 At this stage, the circles and arcs system is not rectangular
since the arcs
connecting $b$
and the bottom line segment of the boundary of the rectangle bounding $X'_b$
are neither of type I nor of type $\sqcup$. 
 We will deform them into arcs of type I one by one from now.

\begin{figure}[htbp]
\begin{center}
\includegraphics[width=120mm]{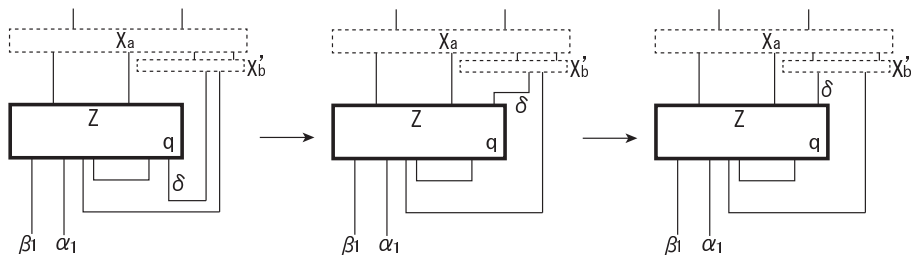}
\end{center}
\caption{The case where $\delta$ is outside $Z$ and of type I}
\label{fig:outsideZ2}
\end{figure}

 Let $q$ be the rightmost endpoint among those of arcs of $A$ in $b_r$.
 If $q$ is an endpoint of an arc inside $Z$,
then we perform one of deformations
as in Figures \ref{fig:DeformTypeI} through \ref{fig:DeformTypeI2} above.
 We consider the case where $q$ is an endpoint of an arc $\delta$ outside $Z$.
 If $\delta$ connects $Z$ and another circle of $C$,
then we can deform $\delta$ to be a vertical line segment connecting $Z$ and $X'_b$
as in Figure \ref{fig:outsideZ2}.
 The last deformation in this figure is similar 
to the straight merge described in the proof of Lemma \ref{lemma:s-merge}. 

\begin{figure}[htbp]
\begin{center}
\includegraphics[width=90mm]{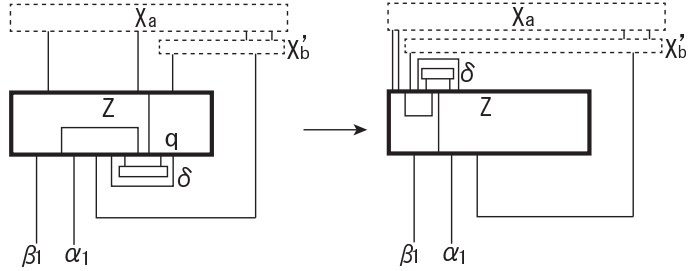}
\end{center}
\caption{The case where $\delta$ is outside $Z$ and of type $\sqcup$}
\label{fig:outsideZ3}
\end{figure}

 When $\delta$ has its both endpoints in $Z$ and is of type $\sqcup$,
a subarc of $b_r$ and $\delta$ cobound a disk, say $R_{\sqcup}$. 
 We shrink $R_{\sqcup}$ to be very small,
and move it along a subarc of $b_r$ and the right vertical line segment of $Z$
as in Figure \ref{fig:outsideZ3}.
 If the arc $R_{\sqcup} \cap b_r$ contains endpoints of arcs of $A$ inside $Z$, 
then we also deform inside $Z$
as in Figures \ref{fig:DeformTypeI} through \ref{fig:DeformTypeI2}.

 Repeating such deformations,
int\,$b_r$ becomes free from endpoints of $A$
and $S$ becomes rectangular again.
 Then, similarly to Figure \ref{fig:ne0},
we lengthen $Z$ to the right direction,
and move $\alpha_1$
so that it forms a vertical line segment above the top horizontal line segment of $Z$.
 This completes the proof of the lemma.
\end{proof}

\begin{proof}
 We prove Theorem \ref{theorem:RCAS}.
 The proof proceeds by induction on the number of arcs of $A$.
 If $A$ is empty, then the theorem is very clear.
 We assume that the theorem holds when $A$ consists of $n-1$ arcs,
and consider the case where $A$ consists of $n$ arcs.
 Let $A'$ be the union of arbitrary $n-1$ arcs of $A$.
 We can deform the circles and arcs system $S=C\cup A$
by an ambient isotopy of ${\mathbb R}^2$
so that $C \cup A'$ is rectangular,
and that the arc $\alpha = A-A'$ is composed
of $m$ vertical line segments and $m-1$ horizontal line segments
for some positive integer $m$.
 Note that $A \cap C\ (= \partial A)$ 
are endpoints of vertical line segments in $A$.
 Hence we can deform the rectangular circles of  $C$ very thin in the vertical direction
so that no pair of rectangular circles of the same depth overlap under $\pi_y$.

\begin{figure}[htbp]
\begin{center}
\includegraphics[width=80mm]{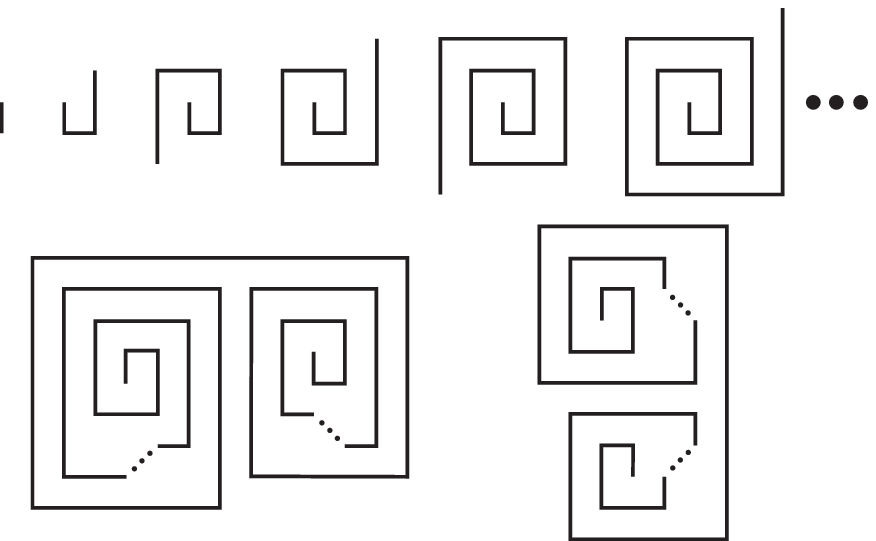}
\end{center}
\caption{$\alpha$}
\label{fig:n_th_arc}
\end{figure}

 If $\alpha$ has a vertical line segment (resp. horizontal line segment) $e$,
which has two horizontal line segments (resp. vertical line segments)
sharing an endpoint with $e$
in both sides of $e$,
then we can perform a deformation
similar to the straight merge operation in the proof of Lemma \ref{lemma:s-merge},
to decrease the number of line segments forming $\alpha$.
 Hence, without loss of generality, we assume
that $\alpha$ does not contain such a line segment
and is of one of the forms shown in Figure \ref{fig:n_th_arc}
or images of them
by a reflection in a vertical line or a horizontal line,
a rotation through $180^{\circ}$ or their composition. 
 Thus, if the arc $\alpha$ has its both endpoints in the same circle of $C$
and is inside the circle,
then $\alpha$ is already of type I or $\sqcup$,
and the theorem follows. 
 Hence we can assume
that either (1) $\alpha$ has its both endpoints in the same circle, say $Z$,
and is outside $Z$,
or (2) $\alpha$ connects distinct circles of $C$.

\begin{figure}[htbp]
\begin{center}
\includegraphics[width=90mm]{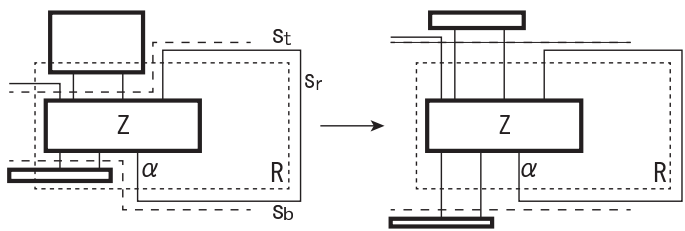}
\end{center}
\caption{get the circles intersecting $\partial R$ out of $R$}
\label{fig:GetCircleOut}
\end{figure}

 We consider first Case (1).
 Then $m=2, 3$ or $4$.
 If $m=2$, then $\alpha$ is of type $\sqcup$, and we are done.
 When $m=3$,
the arc $\alpha$ has two horizontal line segments,
say $s_t$ and $s_b$,
with their ordinates $t_{\alpha}, b_{\alpha}$ satisfying $t_{\alpha} > b_{\alpha}$.
 See Figure \ref{fig:GetCircleOut}. 
 We can assume, without loss of generality,
that the vertical line segment, say $s_r$, in $\alpha$ between $s_t$ and $s_b$
connects right endpoints of $s_t$ and $s_b$.
 Let $\ell_Z$ be the abscissa of the left vertical line segment of $Z$,
$r_{\alpha}$ the abscissa of the vertical line segment $s_r$,
and $R$ the rectangular disk 
$[\ell_Z - \epsilon, r_{\alpha}-\epsilon] \times [b_{\alpha}+\epsilon, t_{\alpha}-\epsilon]$
for a very small positive real number $\epsilon$.
 We take $\epsilon$
so that there is no vertical line segment with its abscissa
in $(\ell_Z - \epsilon, \ell_Z) \cup (r_{\alpha}-\epsilon, r_{\alpha})$,
and there is no horizontal line segment with its ordinate
in $(b_{\alpha}, b_{\alpha}+\epsilon) \cup (t_{\alpha} -\epsilon, t_{\alpha})$.

\begin{figure}[htbp]
\begin{center}
\includegraphics[width=80mm]{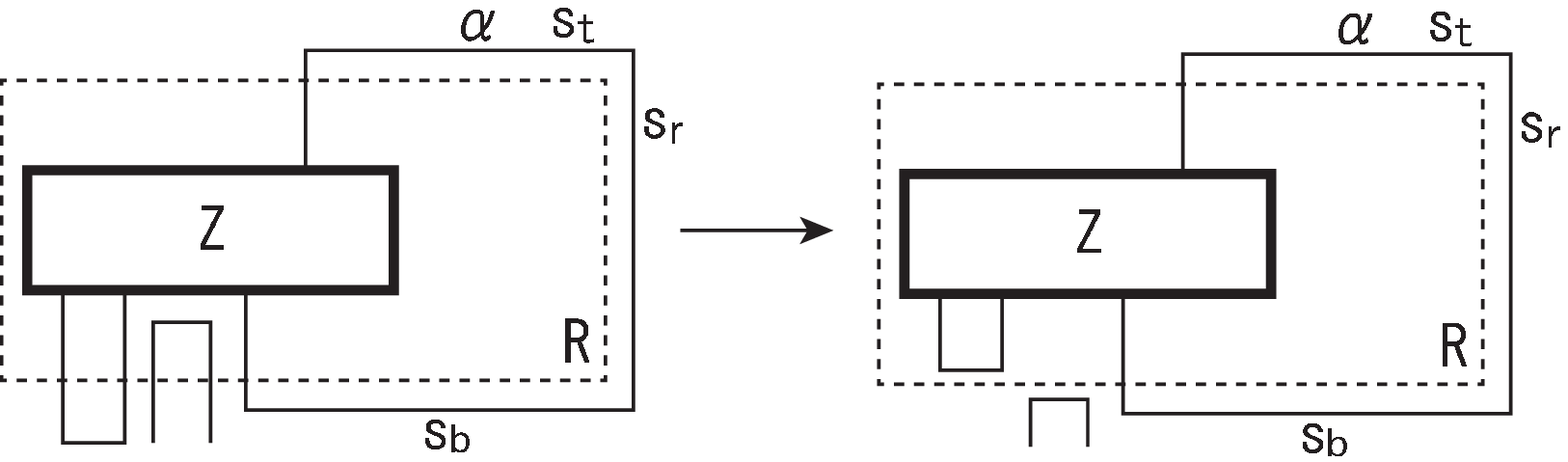}
\end{center}
\caption{get the arcs of type $\sqcup$ away from $\partial R$}
\label{fig:GetRidOfSqcup}
\end{figure}

 If $\partial R$ intersects circles of $C$ other than $Z$
or the left vertical line segment of $\partial R$ intersects horizontal line segments in $A$,
then we perform deformations 
at left parallel copies of the two vertical line segments
$\alpha-(s_b \cup s_r \cup s_t)-{\rm int}\,N(\partial \alpha)$
similar to the straight merge in the proof of Lemma \ref{lemma:s-merge}
to get such circles and horizontal line segments out of $R$,
where $N(\partial \alpha)$ is a very small regular neighborhood of $\partial \alpha$
in $\alpha$.
 See Figure \ref{fig:GetCircleOut}.
 If the bottom horizontal line segment of $\partial R$ intersects arcs of type $\sqcup$,
then we shrink them and what are surrounded by them and subarcs of $C$
in the vertical direction (upward or downward)
to cancel the intersection points.
 See Figure \ref{fig:GetRidOfSqcup}.
 
\begin{figure}[htbp]
\begin{center}
\includegraphics[width=125mm]{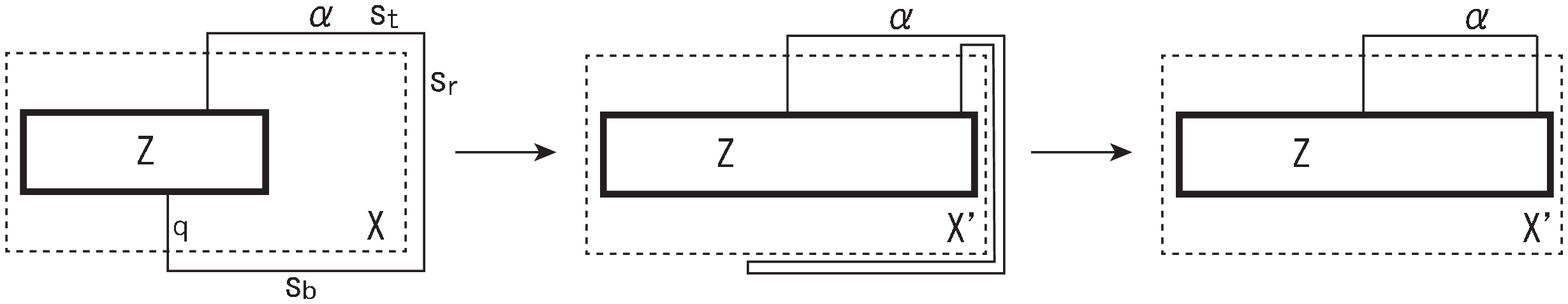}
\end{center}
\caption{Case (1), $m=3$}
\label{fig:Case1m3}
\end{figure}

 Let $q$ be the intersection point of $\alpha$
and the bottom horizontal line segment of $\partial R$.
 We can apply Lemma \ref{lemma:BlackBox} to the black box $X=S \cap R$, 
to bring $q$
to the top horizontal line segment of $\partial R$.
 See Figure \ref{fig:Case1m3}.
 Then we can deform $\alpha$ to an arc of type $\sqcup$,
and the theorem follows in this case.

 When $m=4$, 
a similar argument as above decreases $m$ to $3$,
and then the theorem follows by the above argument.
 See Figure \ref{fig:Case1m4}.

\begin{figure}[htbp]
\begin{center}
\includegraphics[width=125mm]{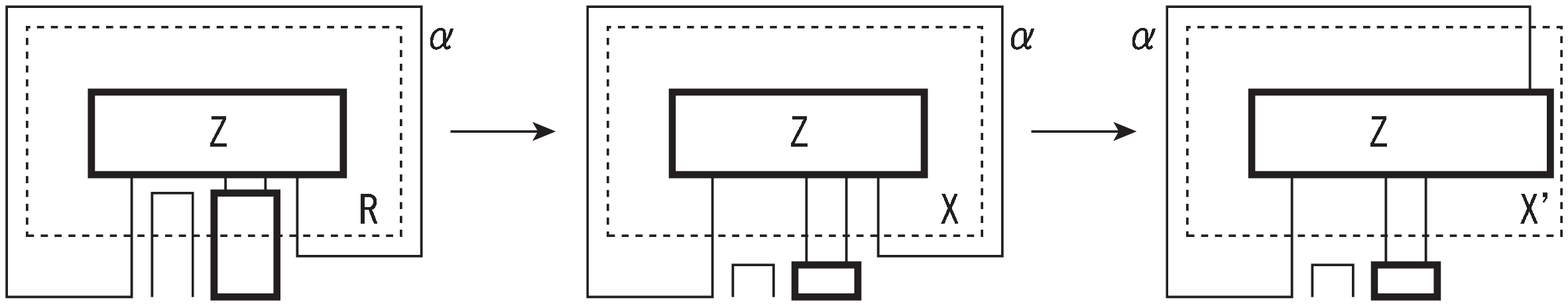}
\end{center}
\caption{Case (1), $m=4$}
\label{fig:Case1m4}
\end{figure}

 We consider Case (2).
 We proceed by induction
on the number $m$ of vertical line segments in $\alpha$.
 If $m=1$, then $S$ is rectangular, and we are done.
 Hence we assume that $m \ge 2$,
and that the theorem holds
if $\alpha$ has at most $m-1$ vertical line segments
as the hypothesis of induction.

 If the two circles of $C$ connected by $\alpha$ are of distinct depths,
then let $Z$ be the circle of larger depth.
 If $\alpha$ connects two circles of $C$ of the same depth,
then we take $Z$ as below.
 When $m \ge 3$,
let $Z$ be the circle of $C$
which contains one of the endpoints of  $\alpha$
whose abscissa is between abscissae of some two vertical line segments in $\alpha$.
 We can assume, without loss of generality,
that $\alpha$ has an endpoint
in the bottom horizontal line segment of $Z$
rather than in the top horizontal line segment.
 When $m=2$,
we can assume, without loss of generality, 
both endpoints of $\alpha$ are contained in the bottom horizontal line segment
of circles of $C$.
 In this case, 
let $Z$ be one of the circles connected by $\alpha$
such that the ordinate of the top horizontal line segment of $Z$
is smaller than
that of the bottom horizontal line segment of the other circle.
 (Recall that circles of the same depth do not overlap under $\pi_y$.)

\begin{figure}[htbp]
\begin{center}
\includegraphics[width=90mm]{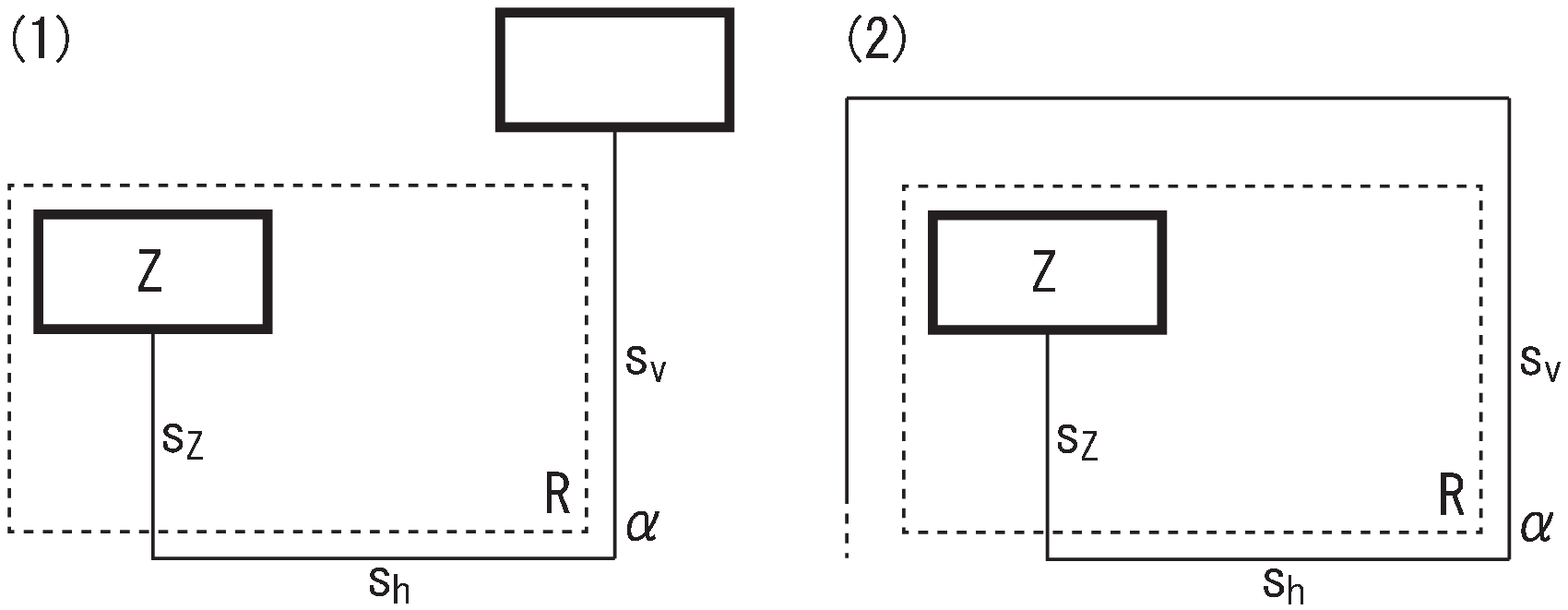}
\end{center}
\caption{Case (2)}
\label{fig:Case2}
\end{figure}

 Let $s_Z$ be the vertical line segment in $\alpha$
with one of its endpoints in $Z$,
$s_h$ the horizontal line segment in $\alpha$
which share an endpoint with $s_Z$,
and $s_v$ the vertical line segment in $\alpha$ other than $s_Z$
such that $s_v$ and $s_h$ share an endpoint.
 We assume, without loss of generality,
that $s_Z \cap s_h$ is the left endpoint of $s_h$.
 Let $x_1$ be the abscissa of the left vertical line segment of $Z$,
$x_2$ the abscissa of the vertical line segment $s_v$,
$y_1$ the ordinate of the horizontal line segment $s_h$,
$y_2$ the ordinate of the top horizontal line segment of $Z$,
and $R$ the rectangular disk
$[x_1-\epsilon, x_2-\epsilon] \times [y_1 + \epsilon, y_2 + \epsilon]$ in ${\mathbb R}^2$
for a very small real number $\epsilon$.
 We take $\epsilon$
so that $S$ has no vertical line segment with its abscissa
in $(x_1 - \epsilon, x_1) \cup (x_2-\epsilon, x_2)$,
and $S$ has no horizontal line segment with its ordinate
in $(y_1, y_1+\epsilon) \cup (y_2, y_2 +\epsilon)$.
 The top horizontal line segment and the right vertical line segment of $\partial R$
do not intersect a circle of $C$.
 See Figure \ref{fig:Case2}.
 If $\partial R$ intersects circles of $C$
or the left vertical line segment of $\partial R$ intersects horizontal line segments in $A$,
then we perform a deformation
at left parallel copy of $s_Z - {\rm int}\,N(\partial \alpha)$
similar to the straight merge as in the proof of Lemma \ref{lemma:s-merge},
to get such circles of $C$ and horizontal line segments out of $R$.
 If the bottom horizontal line segment of $\partial R$
intersects arcs of type $\sqcup$,
then we shrink them and what are surrounded by them and subarcs of $C$
in the vertical direction (upward or downward)
to cancel the intersection points.
 Then we can deform $S$ as shown in Figure \ref{fig:Case2deformation}.
 Precisely,
let $q$ be the intersection point $\alpha \cap \partial R$.
 We can apply Lemma \ref{lemma:BlackBox} to the black box $X=S \cap R$, 
to bring $q$
to the top horizontal line segment of $\partial R$.
 Note that $Z$ is lengthened to the right by the isotopy in the proof of Lemma \ref{lemma:BlackBox}.
 Then we can deform $\alpha$ to an arc with less number of vertical line segments.
 The theorem follows by the hypothesis of induction.
\end{proof}

\begin{figure}[htbp]
\begin{center}
\includegraphics[width=125mm]{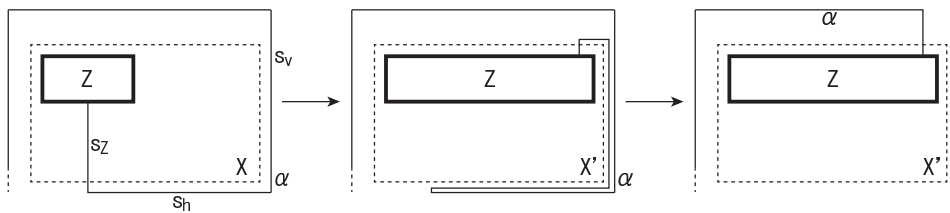}
\end{center}
\caption{deformation in Case (2)}
\label{fig:Case2deformation}
\end{figure}

\section*{Acknowledgments}
 The authors would like to thank Nobuya Satoh for helpful comments.
 He is the advisor of the first author's master's thesis
at Graduate School of Science, Rikkyo University.
 The second author is partially supported
by JSPS KAKENHI Grant Number 25400100.

\end{document}